\documentclass[11pt]{article}
\usepackage{amsmath}
\usepackage{amsfonts}

\newtheorem{theorem}{Theorem}

\newtheorem{definition}[theorem]{Definition}

\newtheorem{proposition}[theorem]{Proposition}
\newtheorem{remark}[theorem]{Remark}

\newenvironment{proof}[1][Proof]{\noindent\textbf{#1.} }{\ \rule{0.5em}{0.5em}}
\input{tcilatex}
\begin{document}

\title{The Bergman kernel for the Vekua equation}
\author{Hugo M. Campos$^{\text{a}}$, Vladislav V. Kravchenko$^{\text{b}}$ \\
$^{\text{a}}${\small School of Mathematical Sciences and Information
Technology,}\\
{\small Yachay Tech, Urcuqui 100119, Ecuador }\\
{\small email: hfernandes@yachaytech.edu.ec }\\
$^{\text{b}}${\small Regional mathematical center of Southern Federal
University, }\\
{\small Bol'shaya vadovaya, 105/42, Rostov-on-Don, 344006, Russian
Federation, }\\
{\small on leave from the Department of Mathematics, CINVESTAV del IPN, }\\
{\small Unidad Queretaro, } {\small Libramiento Norponiente No. 2000, }\\
{\small Fracc. Real de Juriquilla, Queretaro, Qro. C.P. 76230 MEXICO}\\
{\small e-mail: vkravchenko@math.cinvestav.edu.mx\thanks{%
H. M. Campos acknowledges the support by Yachay Tech, Ecuador and CONACYT,
Mexico via the research project CB 2013-222478F. V. V. Kravchenko
acknowledges the support by CONACYT, Mexico via the research project CB
2016-284470.}}}
\maketitle

\begin{abstract}
The paper extends some well-known results for analytic functions onto
solutions of the Vekua equation $\partial \overline{_{z}}W=aW+b\overline{W}$
regarding the existence and construction of the Bergman kernel and of the
corresponding Bergman projection operator.
\end{abstract}

\section{Introduction}

The Bergman kernel is a useful tool in complex analysis with important
applications in differential geometry, in theory of analytic and harmonic
functions and in the study of conformal mappings \cite{Bergman} (see also 
\cite{Krantz2} for new results and extensions to several complex variables).
The idea behind its construction is rather simple. The space of square
integrable analytic functions on a planar domain (called Bergman space) is
associated with a kernel that allows the reproduction of any function in the
Bergman space.

This kind of reproducing property was the base of important generalizations.
On the one hand, S. Bergman and his collaborators developed a method for
solving boundary value problems for elliptic pde's with variable
coefficients, based on the construction of some appropriate reproducing
kernels \cite{Bergman2}. On the other hand, analogous reproducing properties
appeared in the work of other mathematicians of that time (and even earlier)
in the study of other functional spaces. In \cite{Aronszajn} the abstract
theory of reproducing kernel Hilbert spaces (RKHS) was developed. The
Bergman-type kernels are concrete examples of RKHS. For the main results,
applications and historical facts about the RKHS theory see \cite{Aronszajn}%
, \cite{Bell}, \cite{Saitoh1} and \cite{Saitoh2}. Despite the fact that the
RKHS theory is well developed, when a concrete reproducing kernel is found,
special properties of the particular kernel and new connections with
different branches of analysis arise.

The goal of this paper is to establish the existence and some first
properties of the Bergman kernel for the Vekua equation. The paper is
organized as follows. In Section \ref{SectVek} we present some well known
properties of solutions of the Vekua equation (\ref{Vekua 1}) needed
throughout the paper and in Section \ref{SectBergman Sp} we define and study
the corresponding Bergman space $H^{2}(\Omega )$. Contrary to the case of
the Bergman space of analytic functions, $H^{2}(\Omega )$ is a real linear
space (in general, not complex). By introducing a suitable scalar product we
realize that $H^{2}(\Omega )$ is a Hilbert space and establish useful
convergence properties. Using this, we show in Section \ref{SectionExistence}
that the real and the imaginary parts of the pointwise evaluation are
continuous linear functionals on $H^{2}(\Omega )$. This leads first to the
existence of the Bergman kernel in $H^{2}(\Omega )$ and second to a
realization of a reproducing property in the space $H^{2}(\Omega )$.
Finally, in Section \ref{Sect_Project} we provide a method for constructing
the Bergman kernel in terms of any countable orthonormal basis in $%
H^{2}(\Omega )$. Moreover, we prove that such countable orthonormal basis
always exists and provide some suggestions for its construction.
Additionally we obtain a representation for the Bergman projection from $%
L^{2}(\Omega )$ onto $H^{2}(\Omega )$.

\section{Properties of solutions of the Vekua equation\label{SectVek}}

Let $K\subset \mathbb{C}$ be a compact set. A function $W:K\rightarrow 
\mathbb{C}$ is called H\"{o}lder continuous on $K$ with exponent $0<\alpha
<1 $ if there exists $M>0$ such that%
\begin{equation*}
\left\vert W(z_{1})-W(z_{2})\right\vert \leq M\left\vert
z_{1}-z_{2}\right\vert ^{\alpha }
\end{equation*}%
for all $z_{1}$, $z_{2}\in K$. The set of such functions is denoted by $%
C^{\alpha }(K)$. If $\Omega $ is a domain (that is, open and connected) then
we define $C^{\alpha }(\Omega )$ to be the space of functions $W:\Omega
\rightarrow \mathbb{C}$ such that $W\in C^{\alpha }(K)$ for any $K\subset
\Omega $ compact.

Consider the Vekua equation%
\begin{equation}
\partial \overline{_{z}}W=aW+b\overline{W}\text{, \ \ \ in }\Omega \text{,}
\label{Vekua 1}
\end{equation}%
where $\partial \overline{_{z}}=\frac{1}{2}(\partial _{x}+i\partial _{y}$)
and the coefficients $a,b\in C^{\alpha }(\Omega )$ are bounded functions in $%
\Omega $. We shall deal with classical solutions of (\ref{Vekua 1}).

\begin{proposition}[{\protect\cite[pag 7]{BersBook}}]
\label{T_prop}Let $\varphi $ be a bounded measurable function defined on a
bounded domain $\Omega \subset \mathbb{C}$. Consider the integral operator%
\begin{equation*}
(T_{\Omega }\varphi \mathcal{)(}z):\mathcal{=}\frac{1}{\pi }\int_{\Omega }%
\frac{\varphi (\zeta )}{z-\zeta }d\Omega _{\zeta }.
\end{equation*}%
Then

\begin{description}
\item[$(i)$] If $\left\vert \varphi (z)\right\vert \leq M$ in $\Omega $ then 
$\left\vert T_{\Omega }\varphi (z)\right\vert \leq k_{1}M$ for all $z\in 
\mathbb{C}$, $k_{1}$ depending only on the area of $\Omega $.

\item[$(ii)$] $T_{\Omega }\varphi (z)\in C^{\alpha }(K)$ for any compact $%
K\subset \mathbb{C}$ and any $0<\alpha <1$.

\item[$(iii)$] If $\varphi \in C^{\alpha }(\Omega )$, then the first order
partial derivatives of $T_{\Omega }\varphi $ belong to $C^{\alpha }(\Omega )$%
, and $\partial \overline{_{z}}T_{\Omega }\varphi (z)=\varphi (z)$ in $%
\Omega $.
\end{description}
\end{proposition}

\bigskip

\begin{proposition}
\label{prop_int_eq}Let $W$ be a bounded measurable function in a domain $%
\Omega $. Set%
\begin{equation*}
\Phi =W-T_{\Omega }(aW+b\overline{W}).
\end{equation*}%
Then, $W$ is a solution of (\ref{Vekua 1}) in $\Omega $ iff $\Phi $ is
analytic in $\Omega .$
\end{proposition}

\begin{proof}
First note that if $W\in C^{\alpha }(\Omega )$ then from Proposition \ref%
{T_prop}, part $(iii)$, $T_{\Omega }(aW+b\overline{W})$ is continuously
differentiable and 
\begin{equation*}
\partial \overline{_{z}}T_{\Omega }(aW+b\overline{W})=aW+b\overline{W}.
\end{equation*}

Assume that $W\in C^{1}(\Omega )$ (thus, $W\in C^{\alpha }(\Omega )$) is a
solution of (\ref{Vekua 1}). Then application of $\partial \overline{_{z}}$
to $\Phi $ gives $\partial \overline{_{z}}\Phi =\partial \overline{_{z}}%
W-(aW+b\overline{W})=0$. Thus $\Phi $ is analytic. On the other hand suppose
that $\Phi $ is analytic. Then, as $W$ is bounded by Proposition \ref{T_prop}%
, part $(i)$, $T_{\Omega }(aW+b\overline{W})\in C^{\alpha }(\Omega )$. Since 
$\Phi \in C^{\alpha }(\Omega )$ it follows that $W\in C^{\alpha }(\Omega )$.
Using $\partial _{\overline{z}}\Phi =0$ we get 
\begin{equation*}
\partial _{\overline{z}}W=\partial _{\overline{z}}\Phi +\partial _{\overline{%
z}}T_{\Omega }(aW+b\overline{W})=aW+b\overline{W}.
\end{equation*}
\end{proof}

\begin{proposition}
Any solution of (\ref{Vekua 1}) in $\Omega $ has H\"{o}lder continuous
partial derivatives in every compact subdomain of $\Omega $.
\end{proposition}

\begin{proof}
This follows from the above proposition and from the properties of the
operator $T_{\Omega }$ established in Proposition \ref{T_prop}.
\end{proof}

\begin{proposition}
\label{theorem_pseudo_conv}The limit of a uniformly convergent sequence of
solutions of (\ref{Vekua 1}) is a solution of (\ref{Vekua 1}) as well.
\end{proposition}

\begin{proof}
Let $W_{n}(z)$ be a sequence of bounded solutions of (\ref{Vekua 1}) in $%
\Omega $ such that $W_{n}(z)\rightarrow W(z)$ uniformly in $\Omega $. It
follows from Proposition \ref{prop_int_eq} that the functions $\Phi _{n}$
defined by%
\begin{equation*}
\Phi _{n}=W_{n}-T_{\Omega }(aW_{n}+b\overline{W_{n}})
\end{equation*}%
are analytic in $\Omega $. Then the uniform limit $\Phi =W-T_{\Omega }(aW+b%
\overline{W})$ is also an analytic function in $\Omega $. The last equality
together with Proposition \ref{prop_int_eq} allows one to conclude that $W$
is a solution of (\ref{Vekua 1}) in $\Omega $.
\end{proof}

\begin{theorem}[Similarity Principle \protect\cite{BersBook}, \protect\cite%
{VekuaBook}]
\label{Similarity}Let $W$ be a solution of (\ref{Vekua 1}) in $\Omega $
(allowed to be unbounded) and set $S:=T_{\Omega }(a+\frac{\overline{W}}{W}b)$%
. Then there exists an analytic function $\Psi $ in $\Omega $ such that%
\begin{equation*}
W=\Psi e^{S},\quad \text{in }\Omega \text{.}
\end{equation*}
\end{theorem}

\begin{remark}
\label{Remark Similarity}We stress that $a+\frac{\overline{W}}{W}b$ is a
bounded function in $\Omega $ (even if $W$ is unbounded) and thus from part $%
(ii)$ of Proposition \ref{T_prop} it follows that the function $S$ of the
above theorem is H\"{o}lder continuous in every bounded domain. Moreover,
from part $(i)$ of the same proposition it is easy to see that $\left\vert
S(z)\right\vert =k_{1}M$ for all $z\in \mathbb{C}$, where $M=\underset{%
\Omega }{\sup }(\left\vert a(z)\right\vert +\left\vert b(z)\right\vert )$.
This proves that $S$ is bounded by a constant that does not depend on $W$.
\end{remark}

\section{The Bergman space for the Vekua equation\label{SectBergman Sp}}

Let $\Omega $ be a bounded domain and $L^{2}(\Omega )$ the classical
Lebesgue space of complex valued measurable functions $\varphi $ enjoying
the property $\int_{\Omega }\left\vert \varphi (z)\right\vert
^{2}dxdy<\infty $. We define the Bergman space of the Vekua equation (\ref%
{Vekua 1}) as%
\begin{equation*}
H^{2}(\Omega )=\left\{ W\in C^{1}(\Omega )\cap L^{2}(\Omega ):\partial 
\overline{_{z}}W=aW+b\overline{W}\text{ in }\Omega \right\} \text{,}
\end{equation*}%
where $a$, $b\in C^{\alpha }(\Omega )$ are bounded functions in $\Omega $.

It would be convenient to introduce an inner product structure on $%
H^{2}(\Omega )$, regarding it as a subspace of $L^{2}(\Omega )$. However,
since $H^{2}(\Omega )$ is a real linear space it is necessary to consider $%
L^{2}(\Omega )$ as a real linear space as well. The following remark is
useful for what follows.

\begin{remark}
\label{LinearRC}Let $\mathcal{H}_{\mathbb{C}}$ be a complex linear space and 
$\left\langle \cdot ,\cdot \right\rangle _{\mathbb{C}}$ a (complex) inner
product defined there. Denote by $\mathcal{H}_{\mathbb{R}}$ the same set (as 
$\mathcal{H}_{\mathbb{C}}$) but understood as a real linear space. Then, it
is straightforward to check that $\left\langle u,v\right\rangle _{\mathbb{R}%
}:=\func{Re}\left\langle u,v\right\rangle _{\mathbb{C}}$ is a (real) inner
product in $\mathcal{H}_{\mathbb{R}}$. Moreover, both scalar products induce
the same\ norm, that is,%
\begin{equation*}
\left\Vert u\right\Vert _{\mathbb{C}}=\sqrt{\left\langle u,u\right\rangle _{%
\mathbb{C}}}=\sqrt{\left\langle u,u\right\rangle _{\mathbb{R}}}=\left\Vert
u\right\Vert _{\mathbb{R}}\text{ for all }u\in \mathcal{H}_{\mathbb{C}%
}\equiv \mathcal{H}_{\mathbb{R}}.
\end{equation*}%
Also, if $\mathcal{H}_{\mathbb{C}}$ is a complex Hilbert space then $%
\mathcal{H}_{\mathbb{R}}$ is a real Hilbert space.
\end{remark}

From now on, and according to the previous remark, we consider the real
Hilbert space $L^{2}(\Omega )$ with the inner product given by

\begin{eqnarray*}
\left\langle W,V\right\rangle &=&\func{Re}\int_{\Omega }W\overline{V}dxdy \\
&=&\int_{\Omega }\func{Re}W\func{Re}Vdxdy+\int_{\Omega }\func{Im}W\func{Im}%
Vdxdy.
\end{eqnarray*}%
The norm and the inner product in $H^{2}(\Omega )$ are those induced by $%
L^{2}(\Omega )$,%
\begin{equation*}
\left\langle W,V\right\rangle _{H^{2}(\Omega )}=\func{Re}\int_{\Omega }W%
\overline{V}dxdy\text{, \ \ }\left\Vert W\right\Vert _{H^{2}(\Omega
)}:=\left\Vert W\right\Vert _{L^{2}(\Omega )}.
\end{equation*}

If $a=b\equiv 0$ then $H^{2}(\Omega )$ coincides with the classical Bergman
space of analytic functions which we shall denote by $A^{2}(\Omega )$
(however, and contrary to the classical case, in the present paper $%
A^{2}(\Omega )$ is understood as a real space).

The next two propositions generalize corresponding well known results for $%
A^{2}(\Omega )$ \cite{Bell},\cite{Krantz}.

\begin{proposition}
\label{Imersion_acotada}Let $K\subset \Omega $ be a compact. There is a
constant $C_{K}>0$ such that for all $W\in H^{2}(\Omega )$ the following
inequality is valid 
\begin{equation*}
\sup_{z\in K}\left\vert W(z)\right\vert \leq C_{K}\left\Vert W\right\Vert
_{H^{2}(\Omega )}\text{.}
\end{equation*}
\end{proposition}

\begin{proof}
Let $K\subset \Omega $ be a compact and let $W\in H^{2}(\Omega )$. From
Proposition \ref{Similarity} there exist $\Psi $ an analytic function in $%
\Omega $ and $S$ a bounded function in $\mathbb{C}$ such that $W=\Psi e^{S}$%
. It is clear that $\Psi \in A^{2}(\Omega )$. According to Remark \ref%
{Remark Similarity} there exists a constant $M>0$ independent of $W$ such
that $\left\vert e^{S(z)}\right\vert $, $\left\vert e^{-S(z)}\right\vert
\leq M$, for all $z\in \mathbb{C}$. Hence%
\begin{equation*}
\sup_{z\in K}\left\vert W(z)\right\vert \leq M\sup_{z\in K}\left\vert \Psi
(z)\right\vert \text{.}
\end{equation*}%
Moreover, according to \cite[Lemma 1.2.1]{Krantz}, \cite{VekuaBook} there
exists a constant $D_{K}>0$, depending on $K$, such that%
\begin{equation*}
\sup_{z\in K}\left\vert \Psi (z)\right\vert \leq D_{K}\left\Vert \Psi
\right\Vert _{L^{2}(\Omega )}\text{, }\forall \text{ }\Psi \in A^{2}(\Omega
).
\end{equation*}%
Combining the last inequalities we obtain the desired result%
\begin{equation*}
\sup_{z\in K}\left\vert W(z)\right\vert \leq MD_{K}\left\Vert \Psi
\right\Vert _{L^{2}(\Omega )}=MD_{K}\left\Vert We^{-S}\right\Vert
_{L^{2}(\Omega )}\leq M^{2}D_{K}\left\Vert W\right\Vert _{L^{2}(\Omega )}.
\end{equation*}
\end{proof}

\begin{proposition}
\label{H2 Hilbert}$H^{2}(\Omega )$ is a Hilbert space. Moreover, the
convergence in $H^{2}(\Omega )$ implies the uniform convergence on all
compact subsets of $\Omega $.
\end{proposition}

\begin{proof}
Let $W_{n}$ be a Cauchy sequence in $H^{2}(\Omega )$ and let $K$ be a
compact subset of $\Omega $. As $W_{n}$ is a Cauchy sequence in the Hilbert
space $L^{2}(\Omega )$ there exists $W\in L^{2}(\Omega )$ such that $%
W_{n}\rightarrow W$ with respect to the $L^{2}$ norm. Using this together
with Proposition \ref{Imersion_acotada} we see that $W_{n}$ is a Cauchy
sequence in the Banach space $C(K)$. Thus, $W_{n}$ converges uniformly on $K$
to some function $\widetilde{W}\in C(K)$. It follows from Proposition \ref%
{theorem_pseudo_conv} that $\widetilde{W}$ is a solution of the Vekua
equation. In order to finish the proof let us show that $W=\widetilde{W}$
almost everywhere in $K$. It is easy to see from the above reasoning that $%
W_{n}\rightarrow W$ in $L^{2}(K)$ as well as $W_{n}\rightarrow \widetilde{W}$
in $L^{2}(K)$. Thus, $W=\widetilde{W}$ as elements of $L^{2}(K)$ and the
proposition is proved.
\end{proof}

\section{The Bergman kernel for the Vekua equation\label{SectionExistence}}

Let $\zeta \in \Omega $. From Proposition \ref{Imersion_acotada} it follows
that both functionals $\mathcal{R}_{\zeta }$, $\mathcal{I}_{\zeta }$ acting
from $H^{2}(\Omega )$ to $\mathbb{R}$ as $\mathcal{R}_{\zeta }W:=\func{Re}%
W(\zeta )$, $\mathcal{I}_{\zeta }W:=\func{Im}W(\zeta )$ are continuous.
Thus, by Riesz' representation theorem there are functions $K(\zeta
,z)\equiv k_{\zeta }(z)\in H^{2}(\Omega )$ and $L(\zeta ,z)\equiv l_{\zeta
}(z)\in H^{2}(\Omega )$ such that%
\begin{equation}
\func{Re}W(\zeta )=\left\langle W(z),K(\zeta ,z)\right\rangle _{H^{2}(\Omega
)}  \label{rep1}
\end{equation}%
and%
\begin{equation}
\func{Im}W(\zeta )=\left\langle W(z),L(\zeta ,z)\right\rangle _{H^{2}(\Omega
)}  \label{rep2}
\end{equation}%
for all $W\in H^{2}(\Omega )$. The kernels $K(\zeta ,z)$ and $L(\zeta ,z)$
enjoy the following interesting relations.

\begin{proposition}
\label{Prop_KL_rel}For any $\zeta ,z\in \Omega $ the equalities hold%
\begin{equation*}
\func{Re}K(\zeta ,z)=\func{Re}K(z,\zeta ),\text{ \ }\func{Im}L(\zeta ,z)=%
\func{Im}L(z,\zeta )
\end{equation*}%
and%
\begin{equation*}
\func{Re}L(\zeta ,z)=\func{Im}K(z,\zeta ).
\end{equation*}
\end{proposition}

\begin{proof}
Let $\zeta _{1}\in \Omega $. Taking $W(z)=K(\zeta _{1},z)$ in (\ref{rep1})
we obtain $\func{Re}K(\zeta _{1},\zeta )=\left\langle K(\zeta
_{1},z),K(\zeta ,z)\right\rangle _{H^{2}(\Omega )}$. Moreover, since $%
H^{2}(\Omega )$ is a real space 
\begin{equation*}
\func{Re}K(\zeta _{1},\zeta )=\left\langle K(\zeta _{1},z),K(\zeta
,z)\right\rangle _{H^{2}(\Omega )}=\left\langle K(\zeta ,z),K(\zeta
_{1},z)\right\rangle _{H^{2}(\Omega )}=\func{Re}K(\zeta ,\zeta _{1}).
\end{equation*}%
The other equalities are proved similarly.
\end{proof}

\bigskip

Gathering equalities (\ref{rep1}) and (\ref{rep2}) we obtain%
\begin{equation*}
W(\zeta )=\left\langle W(z),K(\zeta ,z)\right\rangle _{H^{2}(\Omega
)}+i\left\langle W(z),L(\zeta ,z)\right\rangle _{H^{2}(\Omega )}
\end{equation*}%
\begin{equation*}
=\int_{\Omega }\left\{ (\func{Re}K(\zeta ,z)+i\func{Re}L(\zeta ,z))\func{Re}%
W(z)+(\func{Im}K(\zeta ,z)+i\func{Im}L(\zeta ,z))\func{Im}W(z)\right\} dxdy.
\end{equation*}%
With the aid of Proposition \ref{Prop_KL_rel} this equality can be written
in the form%
\begin{equation}
W(\zeta )=\int_{\Omega }\left\{ K(z,\zeta )\func{Re}W(z)+L(z,\zeta )\func{Im}%
W(z)\right\} dxdy.  \label{RepW1}
\end{equation}

\begin{definition}
We define the Bergman kernel of the Vekua equation with coefficient $\alpha
\in \mathbb{C}$ and center $\zeta \in \Omega $ as 
\begin{equation*}
B(\alpha ,\zeta ,z):=(\func{Re}\alpha )K(\zeta ,z)+\left( \func{Im}\alpha
\right) L(\zeta ,z).
\end{equation*}
\end{definition}

The previous construction leads to the following statement.

\begin{proposition}
The Bergman kernel $B(\alpha ,\zeta ,z)$ belongs to $H^{2}(\Omega )$ in the
variable $z$ and enjoys the reproducing property%
\begin{equation}
W(\zeta )=\int_{\Omega }B(W(z),z,\zeta )dxdy\text{, \ }\zeta \in \Omega
\label{RepW2}
\end{equation}%
for all $W\in H^{2}(\Omega )$.
\end{proposition}

\begin{remark}
Consider the case $a=b=0$, meaning that the Vekua equation defines analytic
functions. Hence if $W\in A^{2}(\Omega )$ then $iW\in A^{2}(\Omega )$ and $%
<W,iV>_{A^{2}(\Omega )}=-<iW,V>_{A^{2}(\Omega )}$. Using this together with (%
\ref{rep1}) and (\ref{rep2}) it is easy to see that 
\begin{equation*}
\func{Re}L(\zeta ,z)=-\func{Re}L(z,\zeta )\text{, \ \ }\func{Im}L(\zeta ,z)=%
\func{Re}K(z,\zeta ).
\end{equation*}%
On the other hand, from the last equalities and Proposition \ref{Prop_KL_rel}
it follows that $K(z,\zeta )=-iL(z,\zeta )$. Substituting this into (\ref%
{RepW1}) gives us the classical Bergman reproducing formula \cite{Bergman}%
\begin{equation*}
W(\zeta )=\int_{\Omega }K(z,\zeta )W(z)dxdy\text{.}
\end{equation*}
\end{remark}

\section{Construction of the Bergman kernel by means of a countable
orthonormal system. The Bergman projection\label{Sect_Project}}

Since $H^{2}(\Omega )$ is a closed subspace of $L^{2}(\Omega )$ then%
\begin{equation}
L^{2}(\Omega )=H^{2}(\Omega )\oplus \left( H^{2}(\Omega )\right) ^{\bot }
\label{L2descomp}
\end{equation}%
and there exists an orthogonal projection $P_{\Omega }$ from $L^{2}(\Omega )$
onto $H^{2}(\Omega )$. $P_{\Omega }$ is named the Bergman projection.

\begin{proposition}
Let $\varphi \in L^{2}(\Omega )$. Then%
\begin{equation*}
\left( P_{\Omega }\varphi \right) (\zeta )=\int_{\Omega }B(\varphi
(z),z,\zeta )dxdy\text{, \ }\zeta \in \Omega \text{.}
\end{equation*}
\end{proposition}

\begin{proof}
From (\ref{rep1}), (\ref{rep2}) and by Proposition \ref{Prop_KL_rel} it is
straightforward to see that%
\begin{equation*}
\int_{\Omega }B(\varphi (z),z,\zeta )dxdy=\left\langle \varphi (z),K(\zeta
,z)\right\rangle _{H^{2}(\Omega )}+i\left\langle \varphi (z),L(\zeta
,z)\right\rangle _{H^{2}(\Omega )}.
\end{equation*}%
Let $Q_{\Omega }$ be the operator defined on $L^{2}(\Omega )$ by the
right-hand side of this equality. According to (\ref{L2descomp}), if $%
\varphi \in L^{2}(\Omega )$ there exist $W\in H^{2}(\Omega )$ and $\psi \in
\left( H^{2}(\Omega )\right) ^{\bot }$ such that $\varphi =W+\psi $. Since $%
\psi $ is orthogonal to both $K(\zeta ,z)$ and $L(\zeta ,z)$ (both kernels
belong to $H^{2}(\Omega )$) then $Q_{\Omega }\psi =0$ and hence 
\begin{equation*}
Q_{\Omega }\varphi =Q_{\Omega }(W+\psi )=Q_{\Omega }W+Q_{\Omega }\psi
=Q_{\Omega }W.
\end{equation*}%
This shows that $Q_{\Omega }^{2}=Q_{\Omega }$ and that the range of $%
Q_{\Omega }$ is $H^{2}(\Omega )$, or equivalently, $Q_{\Omega }$ is the
orthogonal projection from $L^{2}(\Omega )$ onto $H^{2}(\Omega ).$ Hence $%
Q_{\Omega }=P_{\Omega }$ and the proof is finished.
\end{proof}

\bigskip

A subset $M$ of a normed linear space $X$ is called complete if $\limfunc{%
span}M$ is dense in $X$. Obviously, if $\mathcal{H}$ is a Hilbert space then
from a countable and complete subset of $\mathcal{H}$ it is possible to
construct a countable orthonormal basis of $\mathcal{H}$, first removing the
linearly dependent elements and then applying the Gram--Schmidt
orthonormalization process. Moreover, under the notation of Remark \ref%
{LinearRC} it is clear that if $\left\{ \varphi _{n}\right\} _{n\in \mathbb{N%
}}$ is a complete subset of $\mathcal{H}_{\mathbb{C}}$ then $\left\{ \varphi
_{n},i\varphi _{n}\right\} _{n\in \mathbb{N}}$ is a complete subset of $%
\mathcal{H}_{\mathbb{R}}.$

\begin{remark}
If $\Omega $ is a bounded simply connected domain whose boundary is also the
boundary of an infinite region then the system of the usual nonnegative
powers $\left\{ z^{n},iz^{n}\right\} _{n\in \mathbb{N}_{0}}$ is complete in $%
A^{2}(\Omega )$ \cite{Farrel}, \cite{Markushevich}.
\end{remark}

\begin{proposition}
$H^{2}(\Omega )$ has a countable complete subset.
\end{proposition}

\begin{proof}
It is well know that $L^{2}(\Omega )$ has a countable orthonormal basis $%
\left\{ \varphi _{n}\right\} _{n\in \mathbb{N}}$ (see \cite{Reed}). Thus,
given $W\in H^{2}(\Omega )\subset L^{2}(\Omega )$ we can write the
corresponding Fourier series $W=\sum_{n=1}^{\infty }c_{n}\varphi _{n}$ with $%
c_{n}=\left\langle W,\varphi _{n}\right\rangle _{L^{2}(\Omega )}$. Since $%
P_{\Omega }$ is a bounded operator and a projection onto $H^{2}(\Omega )$,
after applying $P_{\Omega }$ to both sides of the above series we get $%
W=P_{\Omega }W=\sum_{n=1}^{\infty }c_{n}P_{\Omega }\varphi _{n}$. This means
that $\limfunc{span}\left\{ P_{\Omega }\varphi _{n}\right\} _{n\in \mathbb{N}%
}$ is dense in $H^{2}(\Omega )$. The statement is proved.
\end{proof}

\begin{proposition}
Let $\left\{ \varphi _{n}\right\} _{n\in \mathbb{N}}$ be an orthonormal
basis of $H^{2}(\Omega )$ and $W\in H^{2}(\Omega )$. Then,%
\begin{equation}
W(z)=\dsum\limits_{n=1}^{\infty }\left\langle W,\varphi _{n}\right\rangle
_{H^{2}(\Omega )}\varphi _{n}(z).  \label{Fourier}
\end{equation}%
The series converges in $H^{2}(\Omega )$ with respect to the variable $z$
and uniformly on compact subsets of $\Omega .$ In particular,%
\begin{equation*}
B(\alpha ,\zeta ,z)=\dsum\limits_{n=1}^{\infty }\left[ \func{Re}\left( 
\overline{\alpha }\varphi _{n}(\zeta )\right) \right] \varphi _{n}(z).
\end{equation*}
\end{proposition}

\begin{proof}
Formula (\ref{Fourier}) is the representation of $W$ as a Fourier series
corresponding to the orthonormal basis $\left\{ \varphi _{n}\right\} $, thus
it is convergent in $H^{2}(\Omega )$. The uniform convergence then follows
due to Proposition \ref{H2 Hilbert}. The last part of the statement follows
from (\ref{Fourier}) taking $W(z)=B(\alpha ,\zeta ,z)$.
\end{proof}

\begin{remark}
An analogue of the Runge theorem from complex analysis is available for the
Vekua equation, replacing the usual powers of complex analysis by a special
countable system of solutions of the Vekua equation called formal powers 
\cite{BersBook}. There are general conditions under which the system of
formal powers can be constructed by a simple algorithm \cite{BersBook}, \cite%
{VladBook}, \cite{CCK}, \cite{CK}. The previous proposition is obviously
related to the Runge theorem due to uniform convergence on compact sets. We
conjecture that \ if $\Omega $ is a Jordan domain then the system of
nonnegative formal powers is complete in $H^{2}(\Omega )$.
\end{remark}

\begin{remark}
The requirement about the smoothness of the coefficients $a$ and $b$ of
equation (\ref{Vekua 1}) was imposed in order to simplify the exposition. If 
$a,b\in L^{p}(\Omega )$, $p>2$, and dealing with weak solutions of (\ref%
{Vekua 1}), Proposition \ref{theorem_pseudo_conv} and Theorem \ref%
{Similarity} (the ones used in Section \ref{SectBergman Sp}) keep to be
valid \cite{VekuaBook}. Moreover, weak solutions of (\ref{Vekua 1}) are H%
\"{o}lder continuous and thus the study of the functionals $\mathcal{R}%
_{\zeta }W:=\func{Re}W(\zeta )$, $\mathcal{I}_{\zeta }W:=\func{Im}W(\zeta )$
makes sense (note that this was the starting point in Section \ref%
{SectionExistence} to establish the existence of the Bergman kernel).
\end{remark}

\end{document}